\documentclass{amsart}

\usepackage{lineno}
\usepackage{amssymb}
\usepackage{amsmath}
\usepackage{amsthm}
\usepackage{mathrsfs}
\usepackage{bbm}
\usepackage{cite}
\usepackage{color}
\usepackage{enumitem}
\usepackage{pgfplots}
\usepackage[utf8]{inputenc}

\usepackage[colorlinks,hypertexnames=false]{hyperref}
\usepackage[msc-links]{amsrefs}

\usepackage{changes}

\usepackage{mathtools}
\usepackage{extarrows}
\usepackage{setspace}

\usepackage{graphicx}

\usepackage{imakeidx}
\makeindex

\hypersetup{
	colorlinks=true,
	linkcolor=black,
	filecolor=black,
	urlcolor=black,
	citecolor=black,
}

\newtheorem{thm}{Theorem}[section]

\newtheorem{defn}[thm]{Definition}

\newtheorem{proposition}[thm]{Proposition}

\newtheorem{remark}[thm]{Remark}

\newcommand{\id}{\mathop{\mathrm{id}}}
\newcommand{\End}{\mathop{\mathrm{End}}}

\numberwithin{equation}{section}

\pgfplotsset{compat=1.18} 

\begin{document}

\title[Slice Dirac-Regular Mappings]{Growth theorems for Slice Dirac-Regular Mappings over Clifford algebras}
\author{Ting Yang}
\email[Ting Yang]{tingy@ustc.mail.edu.cn}
\address{School of Mathematics and Statistics, Anqing Normal University, Anqing, 246133, China}
\author{Xinyuan Dou}
\email[Xinyuan Dou]{douxinyuan@ustc.edu.cn}
\address{Department of Mathematics, University of Science and Technology of China, Hefei 230026, China}

\keywords{Slice Dirac operator; Clifford algebras; representation formula; slice topology; growth theorems}
\thanks{This work was supported by the NNSF of China (12401104) and the Fundamental Research Funds for the Central Universities (WK0010000091).}

\subjclass[2020]{Primary: 30G35; Secondary: 32A30, 32D05}

\begin{abstract}
	In this paper, we define a class of slice Dirac-regular mappings of several variables over Clifford algebras, based on the concept of $O(3)$-stem mappings. We prove that the slice mappings vanish under the slice Dirac operator, which is equivalent to its $O(3)$-stem mappings satisfy the generalized version of the Cauchy-Riemann equation. Moreover, we establish the growth theorem for slice Dirac-regular starlike mappings in the slice cones of Clifford algebras, as well as for slice Dirac-regular k-fold symmetric mappings.
\end{abstract}

\maketitle
\section{Introduction}

  The theory of slice regular functions was initially introduced by Gentili and Struppa for quaternions \cites{gentili2006new14,gentili2007new4}, which generalized the holomorphic functions in one complex variable to the high dimensional and non-commutative case. It was extended to the setting of Clifford algebras \cites{colombo2011slice9,colombo2010extension10} and octonions  \cite{gentili2010regular5}, and then also has applications in the functional calculus for non-commutative operators \cite{sabadini2011noncommutative11}. Later, the slice analysis has been extended to more general case, include real alternative algebras  \cite{ghiloni2011slice15}, left-alternative algebras and real vector space of even dimension \cite{Dou202000425}.
  During this process, slice technique played an important role. As for the case of octonions, we can consider the octonions as a union of complex planes such that the octonions have the book structure:
  $$\mathbb{O}=\bigcup_{I\in \mathbb{S}}\mathbb{C}_I,$$
  where $I$ is the imaginary unit in $\mathbb{O},\ \mathbb{S}$ consists of all imaginary units $I$ of octonions and the complex plane $\mathbb{C}_I= Span_{\mathbb{R}}\{1,I\}$. There are abundant results in slice analysis. Different from the complex structure mentioned above,
  M. Jin and G. Ren \cite{jin2020slice1} initiated the study of the slice Dirac-regular functions over the octonion algebra $\mathbb{O},$ which decomposed the octonions as a union of quaternionic planes:
  $$\mathbb{O}=\bigcup_{I,J\in \mathbb{S},\ I\perp J}\mathbb{H}_{I,J},$$
  where $\mathbb{H}_{I,J}:=Span_{\mathbb{R}}\{1,I,J,IJ\}$ with $I,J\in \mathbb{S}$ and $I\perp J$. Reference \cite{jin2020slice1} established the representation as formula, the Cauchy–Pompeiu integral formula, and the Taylor as well as the Laurent series expansion formulas. 
  Later, Ghiloni \cite{ghiloni2021slice2} used the term slice Fueter-regular function to present slice Dirac-regular function, and proved that slice Fueter-regular functions are standard (complex) slice functions. They obtained the
  Cauchy-type integral formulas of global nature for the slice Fueter-regular functions of one variable, and a version of its Maximum Modulus Principle. 
  
  The goal of this paper is to introduce the slice Diac-regular mappings over the  Clifford algebras $(\mathbb{R}_m)^n$, generalize the growth theorems for subclasses of its normalized univalent mappings.  In addition, we also obtain the growth theorems for slice Dirac-regular k-fold symmetric mappings in the slice cones of Clifford algebras.

  \section{Priliminaries}
We shall begin with a brief discussion of some relevant content about this paper.  The real Clifford algebra $\mathbb{R}_m$ is an universal associative algebra over $\mathbb{R}$ generated by $m$ basis elements $e_1,\dots,e_m,$
subject to the relations 
    $$e_ie_j+e_je_i=-2\delta_{ij},\qquad i,j=1,\dots,m.$$
As a real vector space, $\mathbb{R}_m$ has dimension $2^m.$ Each element       $x\in \mathbb{R}_m$ can be expressed as 
\begin{displaymath}
     x= \sum_{A\in \mathcal{P}(m)}x_Ae_A,
\end{displaymath}
where 
    $$\mathcal{P}(m)=\left\{ \left(h_1,\dots,h_r \right)\in \mathbb{N}^r|\ r=1,\dots,m,1\leq h_1\le \dots\le h_r \leq m \right\}. $$
For each $A=\left\{ \left(h_1,\dots,h_r \right)\in \mathcal{P}    (m)\right\},$ the coefficients $x_A\in \mathbb{R},$ and the products $e_A:=e_{h_1}e_{h_2}\dots e_{h_r}$ are  the basis elements of the  Clifford algebra $\mathbb{R}_m.$ The unit of the Clifford algebra corresponds to $A=\emptyset,$ and we set $e_{\emptyset}=1.$ As usual, we identify the real number field $\mathbb{R}$ with the subalgebra of  $\mathbb{R}_m$ generated by the unit.

Let 
$$\mathbb{S}_m=\{J\in \mathbb{R}_m|\ J^2=-1\}.$$
The elements of $\mathbb{S}_m$ are called the square roots of $-1$ in the Clifford algebra $\mathbb{R}_m.$

Since the non-commutativity of Clifford algebras, we usually consider the subset of  $\left ( \mathbb{R}_m\right)^n$: 
the n-dimensional
 slice cone of Clifford numbers 
$$\left ( \mathbb{R}_m\right)_s^n:=\bigcup_{I\in \mathbb{S}_m}\mathbb{C}_I^n,$$
where $\mathbb{C}_I^n=\mathbb{R}^n+I\ \mathbb{R}^n\ \text{with}\ I\in \mathbb{S}_m$. It provides a suitable way to study the holomorphy on the slice $\mathbb{C}_{I_1}\times\mathbb{C}_{I_2}\dots\times\mathbb{C}_{I_n}.$

\section{Main results}

Let $A=(a_{i,j})$ be a $\ell_1\times \ell_2$ real matrix. Denote by
\begin{equation*}
    A_{\langle n, m\rangle}:=(a_{i,j}\ {\id}_{\mathbb{R}_m^n})\in [\End(\mathbb{R}_m^n)]^{\ell_1\times\ell_2},
\end{equation*}
where
\begin{equation*}
    \mathbb{R}_m^n:=(\mathbb{R}_m)^n.
\end{equation*}
and $\End(\mathbb{R}_m)$ be the class of real linear mapping from $\mathbb{R}_m^n$ to $\mathbb{R}_m^n$.

Denote by $O(3)$ the group of $3\times 3$ orthogonal real matrices. Denote
\begin{equation*}
    O^*(3):=\left\{\begin{pmatrix}
          {\id}_{\mathbb{R}_m^n}\\ & Q_{\langle n,m\rangle}
    \end{pmatrix}:Q\in O(3)\right\}\subset[\End(\mathbb{R}_m^n)]^{4\times 4}.
\end{equation*}
The theory of $O(3)$-stem mappings plays an important role in constructing slice Dirac-regular mappings, we first introduce the domain of the $O(3)$-stem mappings. Denote
\begin{equation*}
    \mathrm{R}^{4n}:=\left\{\begin{pmatrix}x_0\\x_1\\x_2\\x_3\end{pmatrix}:x_\ell\in\mathbb{R}^n \subset \mathbb{R}_m^n,\ \ell=0,1,2,3.\right\}\subset(\mathbb{R}_m^n)^{4\times 1}.
\end{equation*}

{\bf Conventions:}
We regard the elements in $\mathrm{R}^{4n}$ as $4\times 1$ $\mathbb{R}_m^n$-matrices. Moreover, in this paper, the entries of the matrices are regarded either as elements of $\mathbb{R}_m^n$ or as elements of $\End(\mathbb{R}_m^n)$.
\begin{defn}
Let $D \subset \mathrm{R}^{4n}$ be an open set. $D$ is called $O(3)$-invariant, if
\begin{equation*}
    gx\in D,\qquad\forall\ x\in D\quad\mbox{and}\quad g\in O^*(3).
\end{equation*}
\end{defn}

\begin{defn}\label{stem}
	  Let $D \subset \mathrm{R}^{4n}$ be an $O(3)$-invariant open set. A mapping $\mathcal F:D\rightarrow\left (\mathbb{R}_m^n\right)^{4\times1}$
      is $O(3)$-stem, if
	\begin{equation}\label{invariant}
		\mathcal F\left (g x\right )=g \mathcal F(x),\qquad\forall\ x\in D\quad\mbox{and}\quad g\in O^*(3).
	\end{equation}
\end{defn}

In the following paper, it is assumed that the open set $D \subset \mathrm{R}^{4n}$ is $O(3)$-invariant.
A related result which is useful in the study of the property of $O(3)$-stem mappings as follows.
\begin{proposition}\label{orthogonal}
  Let
  \begin{equation*}
      \mathcal F=\begin{pmatrix}
          \mathcal F_0\\\mathcal F_1\\\mathcal F_2\\\mathcal F_3
      \end{pmatrix}:D\rightarrow(\mathbb{R}_m^n)^{4\times 1}
  \end{equation*}
  be $O(3)$-stem with $F_\ell:D\rightarrow \mathbb{R}_m^n, \ell=0,1,2,3$. Then
  \begin{equation}\label{eq-f2f3}
      \mathcal F_2(x)=
   0\quad \text{and} \quad  \mathcal F_3(x)=0,\qquad\forall\ x=\begin{pmatrix}x_0\\x_1\\0\\0\end{pmatrix}\in D,
  \end{equation}
  where $x_0,x_1\in\mathbb{R}^n\subset\mathbb{R}_m^n$.
\end{proposition}

   \begin{proof}
       Suppose $\mathcal{F}$ is an ${O(3)}$-stem mapping on $D$, then by definition, $D$ is O(3)-invariant. Since $x=(x_0,x_1,0,0)\in D$, we take 
   \begin{equation*}
       g=\begin{pmatrix}
    \mathbb{I}_{n\times n}\\ & \mathbb{I}_{n\times n}\\ & & -\mathbb{I}_{n\times n}\\ & & &\mathbb{I}_{n\times n}
    \end{pmatrix},\begin{pmatrix}
    \mathbb{I}_{n\times n}\\ & \mathbb{I}_{n\times n}\\ & & \mathbb{I}_{n\times n}\\ & & &-\mathbb{I}_{n\times n}
    \end{pmatrix}
   \end{equation*}
in equation (\ref{invariant}), respectively. It is easy to deduce that \eqref{eq-f2f3} holds.
\end{proof}

Denote
\begin{equation*}
    \varphi^{I,J}:=({\id}_{\mathbb{R}_m^n},I\ {\id}_{\mathbb{R}_m^n},J\ {\id}_{\mathbb{R}_m^n},IJ\ {\id}_{\mathbb{R}_m^n}).
\end{equation*}
When there is no ambiguity, we abbreviate $\varphi^{I,J}$ as $(1,I,J,IJ)$.

Denote
\begin{equation*}
    [D]:=\{\varphi^{I,J}(x)\in(\mathbb{R}_m)^n_s:\ x\in D\mbox{ and }I,J\in\mathbb{S}_m\mbox{ with } I\perp J\}.
\end{equation*}
In particular, we define
\begin{equation*}
    B_{\mathbb{R}^{4n}}:=\{x\in{R}^{4n}:|x|<1\},
\end{equation*}
and
\begin{equation*}
	\mathbb{B}:=\{\varphi^{I,J}(x)\ \in(\mathbb{R}_m)^n_s:\ x\in D\cap B_{\mathbb{R}^{4n}}\mbox{ and } I,J\in\mathbb{S}_m\mbox{ with } I\perp J\}.
\end{equation*}

Let $I\in \mathbb{S}_m$. Denote
\begin{equation*}
    [D]_I:=[D]\bigcap \mathbb{C}_I^n.
\end{equation*}
It is easy to deduce that the above
 $\mathbb{B}$ is actually a unit ball in the set of several variables in the slice cone of real Clifford algebras $\mathbb{R}_m^n,$ i.e.,
$$\mathbb{B}=\left\{ q\in (\mathbb{R}_m)^n_s: |q|< 1\right\}.$$

Next we introduce the $O(3)$-stem mappings, which generalize the $O(3)$-stem functions of one variable to higher dimensions.

\begin{defn}\label{slice function}
    A mapping $f:[D]\rightarrow\mathbb{R}_m^n$ is called $O(3)$-slice, if there is  $O(3)$-stem $\mathcal{F}:D\rightarrow(\mathbb{R}_m)^{4n}$ such that
    \begin{equation}\label{Slice}
       f\circ\varphi^{I,J}:=\varphi^{I,J}\circ\mathcal F,\qquad\forall\ I,J\in\mathbb{S}_m\quad \mbox{with}\quad I\perp J.
    \end{equation}
    Denote by $\mathcal{S}_{O(3)}([D],  \mathbb{R}_m^n)$ the class of $O(3)$-slice mappings from $[D]$ to $  \mathbb{R}_m^n$.
\end{defn}
     
\begin{proposition}
   Let $\mathcal{F}:D\rightarrow(\mathbb{R}_m)^{4n}$ be $O(3)$-stem. There is a unique mapping $f:[D]\rightarrow \mathbb{R}_m^n$ satisfying \eqref{Slice}. Denote
   \begin{equation*}
       \mathcal{I}_{O(3)}(\mathcal{F}):=f.
   \end{equation*}
\end{proposition}

\begin{proof}
   We only need to prove that for each $x,y\in D$, $I,J,I',J'\in\mathbb{S}_m$ with $I\perp J$, $I'\perp J'$ and
   \begin{equation}\label{eq-vij}
       \varphi^{I,J}(x)=\varphi^{I',J'}(y),
   \end{equation}
   then
   \begin{equation}\label{eq-vijf}
       \varphi^{I,J}\circ\mathcal F(x)=\varphi^{I',J'}\circ\mathcal F(y).
   \end{equation}
   Write
   \begin{equation*}
        \mathcal{F}=\begin{pmatrix}
           \mathcal F_0\\\mathcal F_1\\\mathcal F_2\\\mathcal F_3
       \end{pmatrix}, \quad
       x=\begin{pmatrix}
           x_0\\x_1\\x_2\\x_3
      \end{pmatrix},\quad
      y=\begin{pmatrix}
           y_0\\y_1\\y_2\\y_3
      \end{pmatrix}
   \end{equation*}
   for some $\mathcal{F}_\ell:D\rightarrow\mathbb{R}_m^n,$ and $x_\ell,y_\ell\in\mathbb{R}^n, \ \ell=0,1,2,3$. Equation \eqref{eq-vij} implies that
   \begin{equation}\label{eq-xy}
       x_0=y_0\qquad\mbox{and}\qquad x_1I+x_2J+x_3IJ=y_1I'+y_2J'+y_3I'J'=rH,
   \end{equation}
   for some $r\in\mathbb{R}^n$ and $H\in\mathbb{S}_m$.
   It is easy to check that there is $J_1\in\mathbb{S}_m$ with $J_1\perp H$ and $g\in O^*(3)$ such that
   \begin{equation*}
       (1,I,J,IJ)g=(1,H,J_1,HJ_1).
   \end{equation*}
   Then
   \begin{equation*}
       \begin{split}
           (1,H,J_1,HJ_1)\begin{pmatrix}x_0\\r\\0\\0\end{pmatrix}= & x_0+rH=(1,I,J,IJ) x\\=&(1,I,J,IJ) g g^Tx=(1,H,J_1,HJ_1)g^Tx.
       \end{split}
   \end{equation*}
   According to $1,H,J_1,HJ_1$ are pairwise orthogonal,
   \begin{equation}\label{eq-gtx}
       g^T x=\begin{pmatrix}x_0\\r\\0\\0\end{pmatrix}=:v.
   \end{equation}
   By Proposition \ref{orthogonal}, $\mathcal{F}_2(v)=\mathcal{F}_3(v)=0$. Then
   \begin{equation*}
       \begin{split}
           \varphi^{I,J}\circ\mathcal F(x)=&(1,I,J,IJ)\mathcal{F}(gg^{T}x)=(1,I,J,IJ)g\mathcal F(g^{T}x)\\=&(1,H,J_1,HJ_1)\mathcal F(v)=\mathcal{F}_0(v)+H\mathcal{F}_1(v).
       \end{split}
   \end{equation*}
   Similarly,
   \begin{equation*}
       \varphi^{I',J'}\circ\mathcal{F}(y)=\mathcal{F}_0(v)+H\mathcal{F}_1(v).
   \end{equation*}
   Therefore, \eqref{eq-vijf} holds.
\end{proof}

\begin{remark}
In the case of $n=1,$ the above mapping $f$ reduces to an $O(3)$-slice function in one Clifford variable, similar to the case of \cite{jin2020slice1} in one octonionic vairable.
\end{remark}

Denote
$$\mathcal{S}_{O(3)}^1([D],\mathbb{R}_m^n):=\left\{
     f=\mathcal{I}_{O(3)}(\mathcal F):[D]\rightarrow \mathbb{R}_m^n\ \text{with}\ \mathcal{F}\in C^1(D)
     \right \}.$$

For a $O(3)$-slice mapping  $f$, we have the following representation formula: 
\begin{proposition}
    Let $f\in \mathcal{S}_{O(3)}([D],\mathbb{R}_m^n),$ let $ I,J,I',J'\in \mathbb{S}_m$ with $I\perp J,I'\perp J',$ then 
    \begin{equation}\label{representation}
        f\circ \varphi^{I',J'}=\varphi^{I',J'}\circ \left(\frac{1}{4}\left( \begin{array}{cccc}
    1&1&1&1\\
   -I&-I&I&I\\
     -J&J&-J&J\\
    -IJ&IJ&IJ&-IJ
    \end{array}\right)
    \left( \begin{array}{cccc}
    f\circ \varphi^{I,J}\\
   f\circ \varphi^{I,-J}\\
     f\circ  \varphi^{-I,J}\\
    f\circ \varphi^{-I,-J}
    \end{array}\right)
    \right).
    \end{equation}
\end{proposition}

\begin{proof}
   Let $ I,J,I',J'\in \mathbb{S}_m$ with $I\perp J$ and $I'\perp J',$   for each $x\in D$, we have $$\varphi^{I,J}(x),\ \varphi^{I',J'}(x)\in [D].$$ Because $D$ is $O(3)$-invariant,
    then $$\varphi^{I,-J}(x),\varphi^{-I,J}(x),\varphi^{-I,-J}(x)\in [D].$$
    Since
    $f\in \mathcal{S}_{O(3)}([D],\mathbb{R}_m^n),$ by equation (\ref{Slice}),
\begin{equation*}
    \left( \begin{array}{cccc}
     f\circ \varphi^{I,J}\\
   f\circ \varphi^{I,-J}\\
     f\circ \varphi^{-I,J}\\
     f\circ \varphi^{-I,-J}
    \end{array}\right)=  \left( \begin{array}{cccc}
    1&I&J&IJ\\
   1&I&-J&-IJ\\
     1&-I&J&-IJ\\
   1&-I&-J&IJ
    \end{array}\right) \left( \begin{array}{cccc}
    \mathcal{F}_0 \\
   \mathcal{F}_1 \\
     \mathcal{F}_2 \\
    \mathcal{F}_3 
    \end{array}\right).
\end{equation*}
We can obtain the solution of the $O(3)$-stem mapping $\mathcal F$:
 \begin{equation*}
  \left( \begin{array}{cccc}
    \mathcal{F}_0 \\
   \mathcal{F}_1 \\
     \mathcal{F}_2 \\
    \mathcal{F}_3 
    \end{array}\right)=   \left( \begin{array}{cccc}
    1&I&J&IJ\\
   1&I&-J&-IJ\\
     1&-I&J&-IJ\\
   1&-I&-J&IJ
    \end{array}\right) ^{-1}  \left( \begin{array}{cccc}
     f\circ \varphi^{I,J}\\
   f\circ \varphi^{I,-J}\\
     f\circ \varphi^{-I,J}\\
     f\circ \varphi^{-I,-J}
    \end{array}\right)
 \end{equation*}

 \begin{equation}\label{Stem-presentaton}
 = \frac{1}{4}\left( \begin{array}{cccc}
    1&1&1&1\\
   -I&-I&I&I\\
     -J&J&-J&J\\
    -IJ&IJ&IJ&-IJ
    \end{array}\right)
    \left( \begin{array}{cccc}
     f\circ \varphi^{I,J}\\
   f\circ \varphi^{I,-J}\\
     f\circ \varphi^{-I,J}\\
     f\circ \varphi^{-I,-J}
    \end{array}\right).    
 \end{equation}
 On the other hand, according to equation (\ref{Slice}),  we get
 \begin{equation}\label{repre}
f\circ \varphi^{I',J'}= \varphi^{I',J'} \circ \mathcal{F}=\varphi^{I',J'} \circ \left( \begin{array}{cccc}
    \mathcal{F}_0 \\
   \mathcal{F}_1 \\
     \mathcal{F}_2 \\
    \mathcal{F}_3 
    \end{array}\right).
 \end{equation}
 Using (\ref{Stem-presentaton}) and (\ref{repre}), we obtain 
 (\ref{representation}) as desired.
\end{proof}

In analogy with the case of one variable function, we give the following definition:  
\begin{defn}
    Let $f\in \mathcal{S}_{O(3)}^1([D],\mathbb{R}_m^n)$ and its $O(3)$-stem mapping $\mathcal{F}=\begin{pmatrix}\mathcal{F}_0\\\mathcal{F}_1\\\mathcal{F}_2\\\mathcal{F}_3\end{pmatrix}.$
    If for each $x=\begin{pmatrix}x_0\\x_1\\x_2\\x_3\end{pmatrix}\in D,$
\renewcommand{\arraystretch}{1.5}

\begin{equation} \label{C-R}
   \left\{\begin{array}{cc}
       \frac{\partial\mathcal{F}_0}{\partial x_0}- \frac{\partial\mathcal{F}_1}{\partial x_1}-\frac{\partial\mathcal{F}_2}{\partial x_2}-\frac{\partial\mathcal{F}_3}{\partial x_3}=0 \\
       \frac{\partial\mathcal{F}_0}{\partial x_1}+ \frac{\partial\mathcal{F}_1}{\partial x_0}-\frac{\partial\mathcal{F}_2}{\partial x_3}+\frac{\partial\mathcal{F}_3}{\partial x_2}=0\\
        \frac{\partial\mathcal{F}_0}{\partial x_2}+ \frac{\partial\mathcal{F}_1}{\partial x_3}+\frac{\partial\mathcal{F}_2}{\partial x_0}-\frac{\partial\mathcal{F}_3}{\partial x_1}=0\\
        \frac{\partial\mathcal{F}_0}{\partial x_3}-\frac{\partial\mathcal{F}_1}{\partial x_2}+\frac{\partial\mathcal{F}_2}{\partial x_1}+\frac{\partial\mathcal{F}_3}{\partial x_0}=0
    \end{array}\right .
\end{equation}
    Then $f$ is called a slice Dirac-regular mapping on $[D]$, and denotes that $f\in \mathcal{SR}_{O(3)}([D]).$
\end{defn}
In the case of $n=1,$ we know that $f$ is a slice Dirac-regular function, similar to the case of octonionic in \cite{jin2020slice1}. 
Moreover, using the property of the $O(3)$-stem mappings, we obtain the equivalent condition for slice Dirac-regular mappings. 
\begin{proposition}\label{D-generalize}
   Let a mapping $f\in \mathcal{S}_{O(3)}^1([D],\mathbb{R}_m^n)$. Then $f\in \mathcal{SR}_{O(3)}([D])$ if and only if  there exist $I,J\in \mathbb{S}_m$ with $I\perp J$ such that
    for each $ x=\begin{pmatrix}x_0\\x_1\\x_2\\x_3\end{pmatrix}\in D,$  
   $$\left(\frac{\partial}{\partial x_0}+I \frac{\partial}{\partial x_1}+J\frac{\partial}{\partial x_2}+IJ\frac{\partial}{\partial x_3}\right)\left(f\circ \varphi^{I,J}(x)\right)=0.$$  
\end{proposition}

\begin{proof}
Let $f\in \mathcal{S}_{O(3)}^1([D],\mathbb{R}_m^n),$ there is an $O(3)$-stem mapping $\mathcal F=\begin{pmatrix}\mathcal{F}_0\\\mathcal{F}_1\\\mathcal{F}_2\\\mathcal{F}_3\end{pmatrix}$ such that $f=\mathcal{I}_{O(3)}(\mathcal F)$. For any $x=\begin{pmatrix}
    x_0\\x_1\\x_2\\x_3
\end{pmatrix}\in D$ and  $I,J\in \mathbb{S}_m$ with $I\perp J,$ we have 
$$f\circ \varphi^{I,J}(x)=\varphi^{I,J}\circ \mathcal{F}(x)=\mathcal{F}_0(x)+I\mathcal{F}_1(x) +J\mathcal{F}_2(x) +IJ\mathcal{F}_3(x). $$
Let  $\mathbb{I}:=(1,I,J,IJ)$ and slice Dirac operator $$D_{\mathbb{I}}:=\partial x_0+I \partial x_1+J\partial x_2+IJ\partial x_3,$$
here 
$$\partial x_{\ell}:=\frac{\partial}{\partial x_{\ell}}\quad \text{for}\ \ell=0,1,2,3.$$
Then
\begin{equation}\label{Dirac operator}
    \begin{split}
        &D_{\mathbb{I}}(f\circ \varphi^{I,J}(x))\\
        =&\left(\partial x_0+I \partial x_1+J\partial x_2+IJ\partial x_3\right) \left ( \mathcal{F}_0(x)+I\mathcal{F}_1(x) +J\mathcal{F}_2(x) +IJ\mathcal{F}_3(x) \right )\\
       =&\left \{\left(\partial x_0,- \partial x_1,-\partial x_2,-\partial x_3\right)+I\left(\partial x_1,\partial x_0,-\partial x_3,\partial x_2\right)\right.
       \\
       &+\left. J\left(\partial x_2,\partial x_3,\partial x_0,- \partial x_1\right)+IJ\left(\partial x_3,-\partial x_2,\partial x_1,\partial x_0\right)\right\} \begin{pmatrix}
           \mathcal{F}_0(x)\\ \mathcal{F}_1(x)\\ \mathcal{F}_2(x)\\ \mathcal{F}_3(x)
       \end{pmatrix}\\
       =&( \nabla g_1^T+I  \nabla g_2^T+J \nabla g_3^T+IJ \nabla g_4 ^T ) \mathcal{F}(x)\\
       =& \nabla g_1^T\mathcal{F}(x)+I  \nabla g_2^T\mathcal{F}(x)+J \nabla g_3^T\mathcal{F}(x)+IJ \nabla g_4 ^T\mathcal{F}(x).
  \end{split}
\end{equation}
If $f\in \mathcal{SR}_{O(3)}([D]),$ then equation $(\ref{C-R})$ holds, i.e., $$\nabla g_\ell^T\mathcal{F}(x)=0,\quad \forall x\in D\ \text{and}\ \ell=1,2,3,4.$$
 Thus $D_{\mathbb{I}}(f\circ \varphi^{I,J}(x))=0$ on $D$ due to equation $(\ref{Dirac operator}).$

Conversely, if $D_{\mathbb{I}}(f\circ \varphi^{I,J}(x))=0$ on $D,$ the following proves that $\nabla g_\ell^T\mathcal{F}(x)=0$ for $\ell=1,2,3,4$. In fact, let $ x\in D,$ 
there is $r\in \mathbb{R}^n$ and $H\in \mathbb{S}_m$ such that $\varphi^{I,J}(x)=x_0+Hr$. 
It is easy to check that there is $$g=\begin{pmatrix}
   \mathbb{I}_{n\times n}\\ & \lambda_1 \mathbb{I}_{n\times n} &* & * \\ & \lambda_2 \mathbb{I}_{n\times n} & * & *\\ & \lambda_3 \mathbb{I}_{n\times n} & * & *\end{pmatrix}\in O^*(3)$$ with $\lambda_\ell\in\mathbb{R}, \ell =1,2,3$ such that \eqref{eq-gtx} holds, i.e.
\begin{equation*}
   \begin{pmatrix}
       x_0\\x_1\\x_2\\x_3
   \end{pmatrix}=g\begin{pmatrix} x_0\\r\\0\\0\end{pmatrix}=
   \begin{pmatrix}
   \mathbb{I}_{n\times n}\\ & \lambda_1 \mathbb{I}_{n\times n} &* & * \\ & \lambda_2 \mathbb{I}_{n\times n} & * & *\\ & \lambda_3 \mathbb{I}_{n\times n} & * & *\end{pmatrix}
   \begin{pmatrix} x_0\\r\\0\\0\end{pmatrix}.
\end{equation*}
Then
\begin{equation*}
    x_\ell=\lambda_\ell r,\qquad\ell=1,2,3.
\end{equation*}
Denote
\begin{equation*}
    [x_\ell;r]:=\lambda_\ell
\end{equation*}
and $v:=\begin{pmatrix}x_0\\r\\0\\0\end{pmatrix}.$ 
Let 
$$g_1=\begin{pmatrix}
    \mathbb{I}_{n\times n}\\ & -\mathbb{I}_{n\times n}\\ & & -\mathbb{I}_{n\times n}\\ & & &-\mathbb{I}_{n\times n}
    \end{pmatrix}$$
and    
$$g_2=\begin{pmatrix}
    &\mathbb{I}_{n\times n}&&\\
    \mathbb{I}_{n\times n}\\ 
    & & &-\mathbb{I}_{n\times n}\\ 
    & & \mathbb{I}_{n\times n}
    \end{pmatrix},$$
    
$$g_3=\begin{pmatrix}
    &&\mathbb{I}_{n\times n}&\\
    &&&\mathbb{I}_{n\times n}\\ 
   \mathbb{I}_{n\times n}&&&\\ 
   & -\mathbb{I}_{n\times n}&&
    \end{pmatrix}$$
and
$$g_4=\begin{pmatrix}
    &&&\mathbb{I}_{n\times n}\\
    &&-\mathbb{I}_{n\times n}&\\ 
   &\mathbb{I}_{n\times n}&&\\ 
   \mathbb{I}_{n\times n}&&&
    \end{pmatrix}.$$
Then $g_1,g_2,g_3,g_4$ are the orthogonal   matrices in $[\End(\mathbb{R}_m^n)]^{4\times 4}$.

Let $$\nabla=\left(\partial x_0,\partial x_1, \partial x_2,\partial x_3\right)\quad \text{and}\quad \nabla_1=\left(\partial x_0, \partial r,\ 0\ ,0\ \right).$$ 
By the choice of $g$, we get that $\nabla= \nabla_1 g^T$. 
Using the fact that $\mathcal{F}$ is an $O(3)$-stem mapping on $D,$ we have
\begin{equation} \label{2}
       \mathcal{F}(x)=\mathcal{F}(gv)=g\mathcal{F}(v).
\end{equation}
By the characteristic of $g$, 
\renewcommand{\arraystretch}{1.5}
  \begin{equation}  \label{3}
   g\mathcal{F}(v)= g\left( \begin{array}{cccc}
    \mathcal{F}_0(v)\\
    \mathcal{F}_1(v)\\
    \mathcal{F}_2(v)\\\mathcal{F}_3(v)
    \end{array}\right)\\
    = g\left( \begin{array}{cccc}
    \mathcal{F}_0(v)\\
    \mathcal{F}_1(v)\\
    0\\0
    \end{array}\right)
    =\begin{pmatrix}
   \mathcal{F}_0(v)\\
   [x_1;r] \mathcal{F}_1(v)\\
   [x_2;r] \mathcal{F}_1(v)\\
   [x_3;r] \mathcal{F}_1(v)
    \end{pmatrix}. 
\end{equation}
Here, $\mathcal{F}_2(v)$ and $\mathcal{F}_3(v)$ are equal to zero by Proposition \ref{orthogonal}. 
 Combining equation (\ref{2}) and equation (\ref{3}), we have
 \begin{equation*}
     \mathcal{F}(x) =\begin{pmatrix}
   \mathcal{F}_0(v)\\
   [x_1;r] \mathcal{F}_1(v)\\
   [x_2;r] \mathcal{F}_1(v)\\
   [x_3;r] \mathcal{F}_1(v)
    \end{pmatrix}.
 \end{equation*}
According to equation (\ref{Dirac operator}), we get
$$D_{\mathbb{I}}(f\circ \varphi^{I,J}(x))=( \nabla g_1^T +I  \nabla g_2^T +J \nabla g_3^T +IJ \nabla g_4 ^T)\begin{pmatrix}
   \mathcal{F}_0(v)\\
   [x_1;r] \mathcal{F}_1(v)\\
   [x_2;r] \mathcal{F}_1(v)\\
   [x_3;r] \mathcal{F}_1(v)
    \end{pmatrix}.$$
Note that
$\nabla g_{\ell}^T=\nabla_1g^T g_{\ell}^T=\nabla_1 (g_{\ell}g)^T$ for $\ell =1,2,3,4.$
Then

\begin{equation}\label{Star}
\begin{split}
    &D_{\mathbb{I}}(f\circ \varphi^{I,J}(x))\\
    &=\left(  \nabla_1 (g_1g)^T +I \nabla_1(g_2g)^T  +J\nabla_1 (g_3g)^T  +IJ\nabla_1  (g_4g)^T \right)\begin{pmatrix}
   \mathcal{F}_0(v)\\
   [x_1;r] \mathcal{F}_1(v)\\
   [x_2;r] \mathcal{F}_1 (v)\\
   [x_3;r] \mathcal{F}_1(v)
    \end{pmatrix}. 
\end{split}
\end{equation}
Since $D_{\mathbb{I}}(f\circ \varphi^{I,J}(x))=0$ on $D,$ we consider
$$x=\begin{pmatrix}x_0\\x_1\\x_2\\x_3\end{pmatrix},\begin{pmatrix}x_0\\-x_1\\-x_2\\x_3\end{pmatrix},\ \begin{pmatrix}x_0\\-x_1\\x_2\\-x_3\end{pmatrix}\ \text{and}\ \begin{pmatrix}x_0\\x_1\\-x_2\\-x_3\end{pmatrix}$$
 respectively, denote $$A:=\begin{pmatrix}
   \mathcal{F}_0(v)\\
   [x_1;r]\mathcal{F}_1(v)\\
   [x_2;r]\mathcal{F}_1 (v)\\
   [x_3;r]\mathcal{F}_1(v)
    \end{pmatrix}.$$
By equation $(\ref{Star}),$ it is obvious that
\begin{equation*}
    \left\{\begin{array}{cccc}
      \left(\nabla_1 (g_1g)^T +I \nabla_1(g_2g)^T  +J\nabla_1 (g_3g)^T  +IJ\nabla_1  (g_4g)^T\right)
      A=0   \\
      \left(\nabla_1 (g_1g)^T -I \nabla_1(g_2g)^T  -J\nabla_1 (g_3g)^T  +IJ\nabla_1  (g_4g)^T\right)
      A=0 \\
      \left ( \nabla_1 (g_1g)^T -I \nabla_1(g_2g)^T  +J\nabla_1 (g_3g)^T  -IJ\nabla_1  (g_4g)^T\right)
      A=0 \\
      \left (\nabla_1 (g_1g)^T +I \nabla_1(g_2g)^T  -J\nabla_1 (g_3g)^T -IJ\nabla_1  (g_4g)^T\right)
      A=0 
      \end{array}\right.
\end{equation*}
Hence
\begin{equation*}
   \nabla_1 (g_{\ell}g)^TA=0\ \text{for}\ \ell=1,2,3,4.
\end{equation*}
That is,
\begin{equation*}
 \nabla g_{\ell}^T\mathcal{F}(x)=0\ \text{for}\ \ell=1,2,3,4.   
\end{equation*}
Therefore, equation $(\ref{C-R})$ holds, and then $f$ is a slice Dirac mapping on $[D]$.
\end{proof}

Based on the  characteristic of the $O(3)$-stem mappings, we derive that if an $O(3)$-slice mapping $f$ preserves one slice, then the maximum and minimum of $|f|$ can be obtained on this preserving slice.

\begin{thm}\label{extreme value}
	Let $f\in \mathcal{S}_{O(3)}\left([D],\mathbb{R}_m^n\right)$ and $f\left([D]_{I^{'}}\right)\subset \mathbb{C}_{I^{'}}^n$ for some ${I^{'}}\in \mathbb{S}_m.$ Let $ I,J\in \mathbb{S}_m$ with $I\perp J.$  For all $x=\begin{pmatrix}
	    x_0\\x_1\\x_2\\x_3
	\end{pmatrix}\in D,$ we have
    
\begin{equation*} 
    \max_{x\in D}|f\circ \varphi^{I,J}(x)|=\max\{ |f(x_0+r{I^{'}})|, |f(x_0-r{I^{'}})| \},
\end{equation*}

\begin{equation*} 
   \min_{x \in D}|f\circ \varphi^{I,J}(x)|=\min\{ |f(x_0+r {I^{'}})|, |f(x_0-r{I^{'}})| \},
\end{equation*}
where $r\in \mathbb{R}^n$ such that $\varphi ^{I,J}(x)=x_0+rH$ for some $H\in \mathbb{S}_m.$
\end{thm}

\begin{proof}
   Suppose $f\in \mathcal{S}_{O(3)}\left([D],\mathbb{R}_m^n\right),$ there exists an $O(3)$-stem mapping $\mathcal{F}=\begin{pmatrix}\mathcal{F}_0\\\mathcal{F}_1\\\mathcal{F}_2\\\mathcal{F}_3\end{pmatrix}$ such that $f=\mathcal{I}(\mathcal F).$ Let $x=\begin{pmatrix}
	    x_0\\x_1\\x_2\\x_3
	\end{pmatrix}\in D$ with $\varphi^{I,J}(x)=x_0+rH.$ 
   Let $v=\begin{pmatrix}
	  x_0\\r\\0\\0   
	\end{pmatrix},$ let $ J_1\in \mathbb{S}_m$ with $J_1\perp H,$ then
    $\varphi^{I,J}(x)=x_0+rH=\varphi ^{H,J_1}(v).$
    Since $f\in \mathcal{S}_{O(3)}\left([D],\mathbb{R}_m^n\right),$
   according to equation (\ref{Slice}), 
 \begin{equation*}
  f\circ \varphi^{I,J}(x)=f\circ \varphi^{H,J_1}(v)=\varphi^{H,J_1}\circ \mathcal{F}(v).
\end{equation*}
According to Proposition \ref{orthogonal}, we get
   $\mathcal F(v)=\begin{pmatrix}
       \mathcal F_0(v)\\ 
       \ \mathcal F_1(v)\\0\\0
   \end{pmatrix}.$
Hence 
\begin{equation}\label{def}
 f\circ \varphi^{I,J}(x)=f(x_0+rH)= \mathcal F_0(v)+H \mathcal F_1(v).   
\end{equation}

Notice that $x_0+rI',x_0-rI'\in [D],$ 
by equation (\ref{def}),
\begin{equation*}
  f(x_0+r{I^{'}})=\mathcal F_{0}(v)+{I^{'}}\mathcal F_{1}(v),
\end{equation*}
\begin{equation*}
  f(x_0-r{I^{'}})=\mathcal F_{0}(v)-{I^{'}}\mathcal F_{1}(v).
\end{equation*}
  It follows that 
\begin{equation*}
  \mathcal F_{0}(v)=\frac{1}{2}\left[ f(x_0+r{I^{'}})+f(x_0-r{I^{'}})\right].
\end{equation*}
\begin{equation*}
  \mathcal F_{1}(v)=\frac{{I^{'}}}{2}\left[ f(x_0-r{I^{'}})-f(x_0+r{I^{'}})\right].
\end{equation*}
By assumption, $f$ preserves a slice on $\mathbb{C}_{I^{'}}^n,$  so 
  $$\mathcal F_{0}(v),\ \mathcal F_{1}    (v)\in \mathbb{C}_{I^{'}}^n.$$
  Combining the formula \eqref{def},
  $| f\circ \varphi^{I,J}(x)|$ has the extreme value in $\mathbb{C}_{I^{'}}^n,$\ (see details in \cite{ren2019growth3}.)
\end{proof}

\section{Growth Theorem In The Starlike OR Convex Domains}
As reference \cite{Dou202000426} pointed out, the coverage domain of regular power series is a $\sigma$-ball, which is generally not open under Euclidean topology. To adapt to the structure of slice analysis, we need use the following slice topology.
\begin{defn}
    We define the slice topology on $(\mathbb{R}_m)^n_s:$
    $$ \tau_s((\mathbb{R}_m)^n_s):=\left\{U\subset (\mathbb{R}_m)^n_s:U_I\in \tau (\mathbb{C}^n_I)\ \text{for any}\ I\in \mathbb{S}_m\right\},$$
    where $U_I:=U\cap\mathbb{C}^n_I$ and $\tau (\mathbb{C}^n_I)$ is the Euclidean topology on $\mathbb{C}^n_I.$
\end{defn}
Now we can generalize the growth theorems for subclasses of slice Dirac-regular mappings over Clifford algebras. 
\begin{defn}
    A set $[D]\in (\mathbb{R}_m)_s^n$ is called starlike, if  zero belongs to $[D]$ and the closed line segment joins zero to each point $q\in [D]$ lines entirely in $[D].$
\end{defn}
\begin{defn}
    A set $[D]\in (\mathbb{R}_m)_s^n$ is called slice starlike, if $[D]_I$ is starlike in $\mathbb{C}_I^n$ for some $I\in \mathbb{S}_m.$ 
\end{defn}
\begin{remark}
    By the construction of the set $[D],$ we get that $ [D]$ is starlike if and only if $[D]$ is slice starlike. 
\end{remark}
\begin{defn}
    A set $[D]\in (\mathbb{R}_m)^n$ is called slice convex, if $[D]_I$ is convex in $\mathbb{C}_I^n$ for some $I\in \mathbb{S}_m,$ i.e., $[D]_I$ is starlike with respect to each of its points. 
\end{defn}
\begin{defn}
    A set $[D]\in (\mathbb{R}_m)_s^n$ is called slice circular, if for some $I\in \mathbb{S}_m$ and $z\in [D]_I, \theta \in \mathbb{R},$ we have $e^{I\theta}z\in [D]_I.$
\end{defn}

\begin{defn}
  Let $I\in \mathbb{S}_m$, a mapping $f:\mathbb{B}_I \rightarrow \mathbb{C}_I^n$ is called k-fold symmetric, if 
  $$ e^{-2\pi I/k}f(e^{2\pi I/k}q)=f(q)\quad \text{for  all}\  q\in \mathbb{B}_I,$$
  where $k$ is a positive integer.
\end{defn}

\begin{defn}
     A domain $[D]\in (\mathbb{R}_m)_s^n$ is called slice domain, if 
     
    $(i)\ [D]\cap \mathbb{R}^n\neq \emptyset;$
    
    $(ii) \ [D]_I\ \text{is a domain of}\ \mathbb{C}_I^n \ \text{for any}\  I\in \mathbb{S}_m.$ 
\end{defn}
By \cite{ren2019growth3}*{ Lemma 4.6}, under the slice topology, we obtain the analytic character of $[D]$ by defining functions.

\begin{proposition}
    A slice domain $[D]$ is bounded, starlike and slice circular if and only if there exists a unique continuous function
    $$\rho:(\mathbb{R}_m)_s^n \rightarrow \mathbb{R} $$
    called the defining function of $[D]$, such that 
    \renewcommand\arraystretch{1.5}
\begin{enumerate}
    \item
    $\rho(q)\geq 0$  for each $q \in [D]; \ \rho(q)=0$ if and only if $ q=0$.
    \item 
    $\rho(tq)=|t|\rho(q)\  \text{for each}\ J \in \mathbb{S}_m,\ q \in \mathbb{C}^n_J\ \text{and}\ t\in \mathbb{C}_J.$
    \item
    $(3)\ [D]=\left \{q\in (\mathbb{R}_m)_s^n:\ \rho(q)<1 \right \}.$
\end{enumerate}    
\end{proposition}
    
\begin{remark}\label{defunction}
    Let $I,J,I',J'\in \mathbb{S}_m$ with $I\perp J$ and $I'\perp J'$, If a slice domain $[D]$ is bounded, starlike and slice circular,   then for any $x\in D,$ we have 
    $$\rho\circ \varphi^{I,J}(x)=\rho\circ \varphi^{I',J'}(x).$$
\end{remark}

\begin{proof}     
Let $x\in D,$  then $\varphi^{I,J}(x),\ \varphi^{I',J'}(x)\in [D].$ Thus there exist $r\in \mathbb{R}^n$ and $H_1,H_2 \in \mathbb{S}_m$ such that $$\varphi^{I,J}(x)=x_0+rH_1\quad \text{and}\quad \varphi^{I',J'}(x)=x_0+rH_2.$$
From the proof of Lemma 4.6 in \cite{ren2019growth3}, we can get  $\rho\circ \varphi^{I,J}(x)=\rho\circ \varphi^{I',J'}(x).$
\end{proof}
Finally, we derive the growth theorem for slice Dirac-regular mapping $f$ under some conditions. For convenience, let $f_I$ be the restriction of $f$ to $[D]_I$ for some $I\in \mathbb{S}_m$. 

\begin{thm}
   Let $[D]$ be a bounded,  starlike and slice circular slice domain in $\left(\mathbb{R}_m\right)_s^n,$ its defining function $\rho(x)\in C^1([D])$ except for a lower dimensional set. Let $f\in \mathcal{SR}_{O(3)}\left([D],\mathbb{R}_m^n\right)$, its restriction $f_{I'}$ is a starlike mapping in $\mathbb{C}_{I'}^n$  and $f\left([D]_{I'}]\right)\subset \mathbb{C}_{I'}^n$ for some $I'\in \mathbb{S}_m$. 
   If $f(0)=0,(f_{I'})'(0)=\id_{\mathbb{R}_m^n}$, then for any $x\in D$ and  $I,J\in \mathbb{S}_m$ with $I\perp J,$
   $$\frac{\rho\circ \varphi^{I,J}(x)}{\left(1+\rho\circ \varphi^{I,J}(x)\right)^2}\leq |f\circ \varphi^{I,J}(x)|\leq \frac{\rho\circ \varphi^{I,J}(x)}{\left(1-\rho\circ \varphi^{I,J}(x)\right)^2}$$
or equivalently,
   $$\frac{|\varphi^{I,J}(x)|}{\left(1+\rho\circ \varphi^{I,J}(x)\right)^2}\leq |f\circ \varphi^{I,J}(x)|\leq \frac{|\varphi^{I,J}(x)|}{\left(1-\rho\circ \varphi^{I,J}(x)\right)^2}.$$
These estimates are sharp.
\end{thm}

\begin{proof}
Let $x=\begin{pmatrix}
	    x_0\\x_1\\x_2\\x_3
	\end{pmatrix}\in D,$ there is $r\in \mathbb{R}^n$ and $H\in \mathbb{S}_m$ such that $\varphi^{I,J}(x)=x_0+rH.$ Let $z:=x_0+rI'$ and $\Bar{z}=x_0-rI'.$ 
 Based on the given conditions, we know that $[D]_{I'}$ is a bounded starlike  circular  domain in $\mathbb{C}_{I'}^n$, $\rho(q)$ is a $C^1$
 function on $[D]_{I'}$  except for a lower dimensional set and $f_{I'}$ is a normalized starlike mapping in $\mathbb{C}_{I'}^n$. According to  the classical results in several complex variables \cite{taishun1998growth22}, we have
$$\frac{\rho(z)}{(1+\rho(z))^2}\leq |f(z)|\leq \frac{\rho(z)}{(1-\rho(z))^2}$$
or equivalently,
   $$\frac{|z|}{(1+\rho(z))^2}\leq |f(z)|\leq \frac{|z|}{(1-\rho(z))^2}.$$

Due to the symmetry of $[D]_{I'}$, we also get
$$\frac{\rho(\Bar{z})}{(1+\rho(\Bar{z}))^2}\leq |f(\Bar{z})|\leq \frac{\rho(\Bar{z})}{(1-\rho(\Bar{z}))^2}.$$
By Theorem \ref{extreme value} and Remark \ref{defunction}, we have 
  $$|f\circ \varphi ^{I,J}(x)|\leq \max\{ |f(x_0+rI')|, |f(x_0-rI')| \}\leq \frac{\rho\circ \varphi ^{I,J}(x)}{(1-\rho\circ \varphi ^{I,J}(x))^2}.$$
The inequality on the left can be proved similarly.
\end{proof}

Using the same method as mentioned above, according to the results in several complex variables \cite{graham2003geometric21} and \cite{liu1998growth23} respectively,  we can draw the following refinement conclusions:
\begin{thm}
   Let $f\in \mathcal{SR}_{O(3)}\left(\mathbb{B},\mathbb{R}_m^n\right)$ and $f\left(\mathbb{B}_{I'}\right)\subset \mathbb{C}_{I'}^n$ for some $I\in \mathbb{S}_m$. Moreover, $f(0)=0,(f_{I'})'(0)=\id_{\mathbb{R}_n^m}$ and the  restriction $f_{I'}$ is a  starlike mapping which is k-fold symmetric in $\mathbb{C}_{I'}^n$, 
    then
  $$\frac{|\varphi ^{I,J}(x)|}{(1+|\varphi ^{I,J}(x)|^{k})^{2/k}}\leq |f\circ \varphi ^{I,J}(x)|\leq \frac{|\varphi ^{I,J}(x)|}{(1-|\varphi ^{I,J}(x)|^{k})^{2/k}},\quad \forall x\in\mathbb{B}_{\mathbb{R}^{4n}}.$$
These estimates are sharp. Consequently, $f(\mathbb{B})$ contains a ball centered at zero and of radius $2^{-2/k}.$
\end{thm}
\begin{thm}
   Let $[D]$ be a bounded, slice convex and slice  circular slice domain in $\left(\mathbb{R}_m\right)_s^n,$ its defining function $\rho(x)\in C^1([D])$ except for a lower dimensional set. Let $f\in \mathcal{SR}_{O(3)}\left([D],\mathbb{R}_m^n\right)$, its restriction $f_{I'}$ is a starlike mapping in $\mathbb{C}_{I'}^n$  and $f\left([D]_{I'}]\right)\subset \mathbb{C}_{I'}^n$ for some ${I'}\in \mathbb{S}_m$. If
   \begin{equation*}
       f(0)=0,\qquad\mbox{and}\qquad (f_{I'})'(0)={\id}_{\mathbb{R}_m^n},
   \end{equation*}
   then for any $x\in D,$
   $$\frac{\rho \circ \varphi ^{I,J}(x)}{1+\rho\circ \varphi ^{I,J}(x)}\leq |f\circ \varphi ^{I,J}(x)|\leq \frac{\rho\circ \varphi ^{I,J}(x)}{1-\rho\circ \varphi ^{I,J}(x)}$$
or equivalently,
   $$\frac{|\varphi ^{I,J}(x)|}{1+\rho\circ \varphi ^{I,J}(x)}\leq |f\circ \varphi ^{I,J}(x)|\leq \frac{|\varphi ^{I,J}(x)|}{1-\rho\circ \varphi ^{I,J}(x)}.$$
These estimates are sharp.
\end{thm}

\bibliographystyle{amsplain}
\bibliography{mybibfile}

@article{jin2020slice1,
  title={Slice Dirac operator over octonions},
  author={Jin, M. and Ren, G. and Sabadini, I.},
  journal={Israel Journal of Mathematics},
  volume={240},
  number={1},
  pages={315--344},
  year={2020},
  publisher={Springer}
}

@article{ghiloni2021slice2,
  title={Slice Fueter-regular functions},
  author={Ghiloni, R.},
  journal={The Journal of Geometric Analysis},
  volume={31},
  number={12},
  pages={11988--12033},
  year={2021},
  publisher={Springer}
}

@article{ren2019growth3,
  title={Growth theorems in slice analysis of several variables},
  author={Ren, G. and Yang, T.},
  journal={Advances in Applied Clifford Algebras},
  volume={29},
  number={5},
  pages={101},
  year={2019},
  publisher={Springer}
}

@article{gentili2007new4,
  title={A new theory of regular functions of a quaternionic variable},
  author={Gentili, G. and Struppa, D.C.},
  journal={Advances in Mathematics},
  volume={216},
  number={1},
  pages={279--301},
  year={2007},
  publisher={Elsevier}
}

@article{gentili2010regular5,
  title={Regular functions on the space of Cayley numbers},
  author={Gentili, G. and Struppa, D.C.},
  journal={The Rocky Mountain Journal of Mathematics},
  pages={225--241},
  year={2010},
  publisher={JSTOR}
}

@incollection{colombo2011slice9,
  title={Slice monogenic functions},
  author={Colombo, F. and Sabadini, I. and Struppa, D. C.},
  booktitle={Noncommutative Functional Calculus: Theory and Applications of Slice Hyperholomorphic Functions},
  pages={17--80},
  year={2011},
  publisher={Springer}
}

@article{colombo2010extension10,
  title={An extension theorem for slice monogenic functions and some of its consequences},
  author={Colombo, F. and Sabadini, I. and Struppa, D. C.},
  journal={Israel Journal of Mathematics},
  volume={177},
  number={1},
  pages={369--389},
  year={2010},
  publisher={Springer}
}

@book{sabadini2011noncommutative11,
  title={Noncommutative functional calculus: theory and applications of slice hyperholomorphic functions},
  author={Colombo, F. and Sabadini, I. and Struppa, D. C.},
  volume={289},
  year={2011},
  publisher={Springer Science \& Business Media}
}

@article{gentili2006new14,
  title={A new approach to Cullen-regular functions of a quaternionic variable},
  author={Gentili, G. and Struppa, D. C.},
  journal={Comptes Rendus. Math{\'e}matique},
  volume={342},
  number={10},
  pages={741--744},
  year={2006}
}

@article{ghiloni2011slice15,
  title={Slice regular functions on real alternative algebras},
  author={Ghiloni, R. and Perotti, A.},
  journal={Advances in Mathematics},
  volume={226},
  number={2},
  pages={1662--1691},
  year={2011},
  publisher={Elsevier}
}

@book{graham2003geometric21,
  title={Geometric function theory in one and higher dimensions},
  author={Graham, I. and Kohr, G.},
  year={2003},
  publisher={CRC Press}
}

@article{taishun1998growth22,
  title={The growth theorem for starlike mappings on bounded starlike circular domains},
  author={Liu, T and Ren,G},
  journal={Chinese Annals of Mathematics},
  volume={19},
  number={4},
  pages={401--408},
  year={1998},
  publisher={Shanghai, China: Shanghai Scientific and Technological Literature Pub. House~…}
}

@article{liu1998growth23,
  title={Growth theorem of convex mappings on bounded convex circular domains},
  author={Taishun, L. and Ren, G.},
  journal={Science in China Series A: Mathematics},
  volume={41},
  number={2},
  pages={123--130},
  year={1998},
  publisher={Springer}
}

@article{Dou202000425,
	Title = {A representation formula for slice regular functions over slice-cones in several variables},
    Author = {Dou,X and Ren,G and Sabadini,I.},
	Year = {2023},
	JOURNAL = {Ann. Mat. Pura Appl.},
}

@article{Dou202000426,
	Author = {Dou,X and Ren,G and Sabadini,I.},
	Title = {Extension theorem and representation
    formula in non-axially-symmetric domains for slice regular functions},
	JOURNAL = {J. Eur. Math. Soc.},
    volume={25},
    number={9},
    pages={ 3665–3694},
    Year = {2023},
}
\printindex

\end{document}